\documentclass{elsarticle}
\usepackage{makeidx}  
\usepackage{amsmath}    
\usepackage{amsthm} 
\usepackage{graphicx}
\usepackage{epic,eepic}

\newtheorem{lemma}{Lemma}
\newtheorem{proposition}{Proposition}
\newtheorem{corollary}{Corollary}
\newtheorem{assumption}{Assumption}

\newdefinition{definition}{Definition}
\newtheorem{example}{Example}

\newdefinition{remark}{Remark}

\newcommand{\Ham}{\mathcal{H}}

\newcommand{\opt}[1]{\hat{#1}}

\newcommand{\dd}{{\rm\,d}}

\newcommand{\eps}{\varepsilon}

\biboptions{authoryear}

\begin{document}

\title{On a sufficient condition for infinite horizon\\ optimal control problems}

\author[msu,mephi,kimm]{Anton O. Belyakov}


\address[msu]{Moscow School of Economics, Lomonosov Moscow State University, Moscow, Russia}
\address[mephi]{National Research Nuclear University ``MEPhI'', Moscow, Russia}
\address[kimm]{Central Economic Mathematical Institute of the Russian Academy of Sciences}

\begin{abstract}
New form of sufficient optimality condition is obtained in comparison with the Mangasarian sufficiency theorem. Both finite and infinite values of objective functional are allowed since concepts of overtaking and weakly overtaking optimality are implied. Examples, where new conditions can be applied are presented. The conditions are shown to be both necessary and sufficient, when Hamiltonian is linear with respect to state and control.
\end{abstract}


\maketitle	

\section{Introduction}\label{sec:intro}
Optimal control problems with infinite horizon play important role in economic theory. For instance, the problem of optimal consumption/investment by a household/firm living infinite time is considered in major advanced textbooks on the theory of economic growth, see, e.g. \cite{Acemoglu2009}. 

There are well known the Mangasarian sufficiency theorem, assuming Hamiltonian concavity w.r.t. state and control variables, and the Arrow sufficiency theorem, assuming that control maximized Hamiltonian is concave w.r.t. state variable, see, e.g. \cite{Seierstad1986}. Both theorems prove that control is
overtaking optimal\footnote{Catching up criterion in \cite{Seierstad1986}.} if
\begin{equation}\label{eq:SCoo}
  \liminf_{t\rightarrow\infty} \langle\psi(t),\opt{x}(t)-x(t)\rangle \geq 0,
\end{equation}
and weakly overtaking optimal\footnote{Sporadically catching up criterion in \cite{Seierstad1986}.} if
\begin{equation}\label{eq:SCwoo}
  \limsup_{t\rightarrow\infty} \langle\psi(t),\opt{x}(t)-x(t)\rangle \geq 0,
\end{equation}
for all admissible \emph{state} trajectories $x(\cdot)$, where $\opt{x}$ is the \emph{optimal state variable}, $\psi$ is the corresponding to $\opt{x}$ \emph{adjoint variable}, and brackets $\langle \cdot,\cdot\rangle$ denote scalar product
of two vectors.
These conditions may be implied by usual transversality conditions
\begin{equation}\label{eq:tcPSI}
  \lim_{t\rightarrow\infty} \psi(t) = 0,\qquad 
  \lim_{t\rightarrow\infty} \langle\opt{x}(t),\psi(t)\rangle = 0,
\end{equation}
 under particular assumptions, see \cite{Michel2003}. Similarly one can check that the state and control belong to particular weighted spaces, see \cite{PickenhainLykina2006}.  
 
In contrast to the more general Arrow theorem the proof of the Mangasarian sufficiency theorem could be written without the maximum condition of the Hamiltonian w.r.t. control variable. This allows us to derive new form of sufficient conditions instead of (\ref{eq:SCoo})--(\ref{eq:SCwoo}) under similar concavity condition and almost in the same way as the Mangasarian theorem.
   
In this paper another form of sufficient conditions, is obtained with the use of a Cauchy-type formula with variable upper limit integral  same as in \cite{BelyakovCONDITION,Aseev2017}. This formula solves the adjoint equation, but in my proof it does not have to converge to the adjoint variable of the maximum principle, as time horizon tends to infinity.  

For problems linear in control and state variables the obtained conditions are both necessary and sufficient.

\section{Statement of the problem}
\label{sec:statement}

Let $X$ be a nonempty open
convex subset of $R^n$, $U$ be an arbitrary nonempty set in
$R^m$. Let us consider the following optimal
control problem:
\begin{align}
  & \int_{t_0}^{\infty}g(x(t),u(t),t)\dd t  \label{eq:J}
\rightarrow\max_{u},\\
  & \dot x(t) = f(x(t),u(t),t),\quad x(t_0) = x_0, \label{eq:dx}
\end{align}
where $u(t)\in U$ and exists state variable $x(t)\in X$ for all $t \in [t_0,+\infty)$.  We call such control $u(\cdot)$ and state variable $x(\cdot)$ trajectories \emph{admissible}. 
Functions $f$ and $g$ are differentiable w.r.t. $(x, u)$, and together with these partial derivatives are defined and
locally bounded, measurable in $t$ for every $(x, u)\in X \times U$, and continuous in $(x, u)$
for almost every $t \in [0,\infty)$. 

Improper integral in (\ref{eq:J}) might not converge for any candidate for optimal control $\opt{u}(\cdot)$, i.e. the limit
\begin{equation}\label{eq:Jinfty}
    \lim_{T\rightarrow\infty} J(\opt{u}(\cdot),x_0,t_0,T),
\end{equation}
might fail to exist, or might be infinite, where we introduce the finite time horizon functional:
\begin{equation}\label{eq:JT}
    J(u(\cdot),x_0,t_0,T) = \int_{t_0}^{T}g(x(t),u(t),t)\dd t,
\end{equation}
subject to state equation (\ref{eq:dx}). Thus functional $J$  may be unbounded or oscillating as $T\rightarrow\infty$. So we consider more general   definitions of optimality.
\begin{definition}\label{def:LOO} An admissible control $\opt{u}(\cdot)$ is 
\emph{overtaking optimal} (OO) if for every admissible control $u(\cdot)$ and every scalar $\eps > 0$ there exists time $T=T(\eps,u(\cdot)) > t_0$ 
such that for all $T' \geq T$ holds
$
  J(u(\cdot),x_0,t_0,T') - J(\opt{u}(\cdot),x_0,t_0,T') \leq \eps
$.
\end{definition}
\begin{definition}\label{def:LWOO} An admissible control $\opt{u}(\cdot)$ is weakly
overtaking optimal (WOO) if for every admissible control $u(\cdot)$, scalar $\eps > 0$, and time $T > t_0$ 
one can find $T' = T'(\eps,T,u(\cdot)) \geq T$ such that 
$
  J(u(\cdot),x_0,t_0,T') - J(\opt{u}(\cdot),x_0,t_0,T') \leq \eps
$.
\end{definition}
These two definitions imply that  for all admissible controls $u(\cdot)$ 
$$\limsup_{T\rightarrow\infty} \left(J(u(\cdot),x_0,t_0,T) - J(\opt{u}(\cdot),x_0,t_0,T)\right) \leq 0,$$
for OO $\opt{u}(\cdot)$ in Definition~\ref{def:LOO}  and
$$\liminf_{T\rightarrow\infty} \left(J(u(\cdot),x_0,t_0,T) - J(\opt{u}(\cdot),x_0,t_0,T)\right) \leq 0,$$
for WOO $\opt{u}(\cdot)$ in Definition~\ref{def:LWOO}. It is clear that if $\opt{u}(\cdot)$ is OO, then it is also WOO. When ordinary optimality holds, i.e. finite limit exists  in (\ref{eq:Jinfty}) and for all admissible controls $u(\cdot)$ 
$$\limsup_{T\rightarrow\infty} J(u(\cdot),x_0,t_0,T)  \leq \lim_{T\rightarrow\infty} J(\opt{u}(\cdot),x_0,t_0,T),$$
then $\opt{u}(\cdot)$ is also both OO and WOO.\footnote{
There are many definitions of optimality in the literature, for example, corresponding uniform optimalities, when $T$ and $T'$ do not depend on $u(\cdot)$ in Definitions \ref{def:LOO} and \ref{def:LWOO}, see e.g., \cite{Khlopin2013}. These definitions are stronger and may lead to absence of corresponding optimal solutions.} Moreover, when  limit in (\ref{eq:Jinfty}) is finite, all WWO controls are ordinary optimal.

\section{Optimality conditions}
\label{sec:conditions}

With the use of Lagrange multipliers, scalar $\lambda$ and vector $\psi$, we introduce \emph{Hamilton-Pontryagin} function
\begin{equation}\label{eq:H}
    \Ham(x,u,t,\psi,\lambda) = \lambda\, g(x,u,t) + \langle \psi, f(x,u,t)\rangle
\end{equation}
where brackets $\langle \cdot,\cdot\rangle$ denote scalar product
of two vectors. 
We consider the following Cauchy-type formula 
\begin{equation}\label{eq:Jx}
\opt{J}_x(t,T) :=  Y(t)^{-1}\int_{t}^{T}Y(s)\,\frac{\partial g}{\partial
    x}(\opt{x}(s),\opt{u}(s),s) \dd s 
\end{equation}
as a particular solution of the adjoint equation:
\begin{equation}\label{eq:adjH}
    -\dot\psi(t) = \frac{\partial\Ham}{\partial
    x}(\opt{x}(t),\opt{u}(t),t,\psi(t),\lambda).
\end{equation}
with $Y(\tau)$ being the fundamental matrix of the linear system
\begin{equation}
  \dot y(t) = \left(\frac{\partial f}{\partial
    x}(\opt{x}(t),\opt{u}(t),t)\right)^{\prime}\,y(t),\label{eq:dy}
\end{equation}
where $\left(\frac{\partial f}{\partial
    x}(\opt{x}(t),\opt{u}(t),t)\right)^{\prime}$ is the transposed Jacobian matrix $\frac{\partial f}{\partial
    x}(\opt{x}(t),\opt{u}(t),t)$.

\begin{assumption}\label{a:lim0}
For almost all time instances $t \geq
t_0$ there exists $T_1 = T_1(t) > t_0$ such that
$\Ham(x,u,t,\opt{J}_x(t,T),1)$ is concave in $(x,u)$ for each $T \geq T_1$.
where we take $\opt{J}_x(\tau,T)$ as defined (\ref{eq:Jx}).
\end{assumption}
This assumption differs from Mangasarian concavity assumption because $\opt{J}_x(t,T)$ does not have to be equal to the adjoint variable for which holds the maximum principle, see Example~\ref{ex:Ramsey}. Moreover we do not require in the following lemma the existence of the limit $$\displaystyle\lim_{T\to\infty} \opt{J}_x(t,T),$$  so $\opt{J}_x(t,T)$ can be unbounded (Example~\ref{ex:unbnd}) or oscillating in $T$  (Example~\ref{ex:oscil}).

\begin{lemma}[Sufficient optimality conditions] 
\label{p:limS}
Let Assumption \ref{a:lim0} be fulfilled, then admissible pair $(\hat{u}(\cdot),\hat{x}(\cdot))$.\\
1) is OO if for each admissible control $u(\cdot)$ holds
\begin{align}\label{eq:limUoo}
&\liminf_{T\rightarrow \infty}\int\limits_{t_0}^{T}\!\left\langle\frac{\partial \Ham}{\partial u}(\opt x(t), \opt u(t),t,\opt{J}_x(t,T),1), \opt u(t) - u(t)\right\rangle\!\dd t \geq 0,
\end{align}
2) is WOO if for each admissible control $u(\cdot)$ holds
\begin{align}\label{eq:limUwoo}
&\limsup_{T\rightarrow \infty}\int\limits_{t_0}^{T}\!\left\langle\frac{\partial \Ham}{\partial u}(\opt x(t), \opt u(t),t,\opt{J}_x(t,T),1), \opt u(t) - u(t)\right\rangle\!\dd t \geq 0.
\end{align}
\end{lemma}

\begin{proof} Let us consider any admissible pair $(u(\cdot),x(\cdot))$, i.e. $ u(t)\in U$ and corresponding trajectory $x(t)\in X$ for all $t>t_0$. The corresponding increment in the value of the functional can be written as follows:
\begin{align}
  \Delta J(T) & := J(\opt{u}(\cdot),x_0,t_0,T) - J(u(\cdot),x_0,t_0,T)\nonumber\\
 &=\int\limits_{t_0}^{T}\!\!\left(g(\opt x(t), \opt u(t),t) - g(x(t),u(t),t)\right)\!\dd t  \nonumber\\
&=\int\limits_{t_0}^{T}\!\!\left(\Ham(\opt x(t), \opt u(t),t,\opt{J}_x(t,T),1) - \Ham(x(t),u(t),t,\opt{J}_x(t,T),1)\right)\!\dd t \nonumber\\
& + \int\limits_{t_0}^{T}\!\langle \opt{J}_x(t,T),  f(x(t),u(t),t) - f(\opt x(t), \opt u(t),t)\rangle\dd t.\label{eq:DJT}
\end{align}
Due to concavity Assumption~\ref{a:lim0} we have the following inequality
\begin{align}
& \Ham(x(t), u(t),t,\opt{J}_x(t,T),1) - \Ham(\opt x(t), \opt u(t),t,\opt{J}_x(t,T),1)  \nonumber\\
& \leq \left\langle\frac{\partial \Ham}{\partial u}(\opt x(t), \opt u(t),t,\opt{J}_x(t,T),1), u(t) - \opt u(t)\right\rangle \nonumber\\ 
& + \left\langle\frac{\partial \Ham}{\partial x}(\opt x(t), \opt u(t),t,\opt{J}_x(t,T),1), x(t) - \opt x(t)\right\rangle,\label{ineq:conc}
\end{align}
where $\frac{\partial \Ham}{\partial x}(\opt x(t), \opt u(t),t,\opt{J}_x(t,T),1) = -  \frac{\dd \opt{J}_x}{\dd t}(t,T)$ by definition of $\opt{J}_x(t,T)$ in (\ref{eq:Jx}).
\begin{align*}
  \Delta J(T) & \geq \int\limits_{t_0}^{T}\!\!\left(\left\langle \frac{\dd \opt{J}_x}{\dd t}(t,T),  x(t) - \opt{x}(t) \right\rangle + \left\langle \opt{J}_x(t,T),  \frac{\dd x(t)}{\dd t} - \frac{\dd \opt{x}(t)}{\dd t} \right\rangle\right)\!\dd t  \nonumber\\
& + \int\limits_{t_0}^{T}\!\left\langle\frac{\partial \Ham}{\partial u}(\opt x(t), \opt u(t),t,\opt{J}_x(t,T),1), \opt u(t) - u(t)\right\rangle\!\dd t.
\end{align*}
The first integral is zero 
\begin{align}
& \int\limits_{t_0}^{T}\!\!\left(\left\langle \frac{\dd \opt{J}_x}{\dd t}(t,T),  x(t) - \opt{x}(t) \right\rangle + \left\langle \opt{J}_x(t,T),  \frac{\dd x(t)}{\dd t} - \frac{\dd \opt{x}(t)}{\dd t} \right\rangle\right)\!\dd t  \nonumber\\
& = \int\limits_{t_0}^{T}\!\frac{\dd}{\dd t} \left\langle \opt{J}_x(t,T),  x(t) - \opt{x}(t) \right\rangle\!\dd t  = \left\langle \opt{J}_x(t,T),  x(t) - \opt{x}(t) \right\rangle \Big|_{t_0}^{T} = 0, \label{eq:0}
\end{align}
since $x(t_0) = \opt{x}(t_0)$ and $\opt{J}_x(T,T) = 0$, see definition in (\ref{eq:Jx}).
 \end{proof}
 
Notice that if in Assumption~\ref{a:lim0} we require linearity w.r.t. $(x,u)$ instead of concavity, then inequality (\ref{ineq:conc}) becomes equality and sufficient conditions also become necessary.
\begin{corollary}[Nesessary and sufficient optimality conditions]\label{c:limNS}
Let $\Ham(x,u,t,\psi,1)$ be linear in $(x,u)$ for each $\psi$ and $t$. 
Then admissible pair $(\hat{u}(\cdot),\hat{x}(\cdot))$\\
1) is OO iff for each admissible control $u(\cdot)$ holds inequality (\ref{eq:limUoo}),\\
2) is WOO iff for each admissible control $u(\cdot)$ holds  inequality (\ref{eq:limUwoo}).
\end{corollary}%
 
The following example applies these optimality conditions, when they are both necessary and sufficient due to linearity of the problem and when $\opt{J}_x(\tau,T)$ oscillates in $T$ having no limit as $T\to\infty$. 

\begin{example}[{\cite[Example 1.2]{Carlson1991infinite}} extended to $b > 0$]\label{ex:oscil}
Let us maximize the following integral
$$\max_u \int_0^{\infty} \left(x_2(t) + b u(t)\right)\dd t, \quad \text{ s.t.:  } \left\{\begin{array}{ll@{\quad}l}
	\dot x_1(t) & = x_2(t),  & x_1(0) = 0,\\
	\dot x_2(t) & = u(t)  - x_1(t), & x_2(0) = 0,\\
\end{array}\right.$$
where $b \geq 0$, subject to the system describing a linear oscillator
with bounded control $u(t)\in[-1, 1]$ for all $t\geq 0$.
 
We have the \emph{state-transition matrix} of the linearized system
$$Y^{\prime}(s)\, Y^{\prime}(t)^{-1} = \exp\left(\left(\begin{array}{cc}
	0 & 1\\
	-1 & 0\\
\end{array}\right)(s-t)\right) = \left( 
\begin{array}{cc}
	\cos(s-t) & \sin(s-t)\\
	-\sin(s-t) & \cos(s-t)\\
\end{array}\right)
$$
and vector-function $\opt{J}_x$, defined in (\ref{eq:Jx}), oscillating in $T$:
\begin{align*}
\opt{J}_x(t,T) & =  \int_{t}^{T}Y(t)^{-1} Y(s) \left(\begin{array}{c}
	0\\
	1\\
\end{array}\right) \dd s \\
    & = \int\limits_{t}^{T} \left(\begin{array}{cc}
	\cos(s-t) & -\sin(s-t)\\
	\sin(s-t) & \cos(s-t)\\
\end{array}\right) \left(\begin{array}{c}
	0\\
	1\\
\end{array}\right) \!\dd s =  \left(\begin{array}{c}
	\cos(T-t)-1\\
	\sin(T-t)\\
\end{array}\right).
\end{align*}
Hamilton-Pontryagin function takes the form
$$ 
\Ham(x,u,t,\psi,\lambda) = \lambda \left(x_2 + b u\right) + \psi_1 x_2 + \psi_2 \left(u - x_1\right),
$$ 
and its derivative w.r.t. $u$ at point $(x,u,t,\psi,\lambda) = (\opt x(t), \opt u(t),t,\opt{J}_x(t,T),1)$
$$
\frac{\partial \Ham}{\partial u}(\opt x(t), \opt u(t),t,\opt{J}_x(t,T),1) = b + \sin(T-t).
$$
Control $\opt u(\cdot)$ is OO iff for each admissible control $u(\cdot)$ holds (\ref{eq:limUoo})
$$ 
\liminf_{T\rightarrow \infty}\int\limits_{0}^{T}\!\left(b + \sin(T-t)\right)\left(\opt u(t) - u(t)\right)\!\dd t \geq 0.
$$ 
Control $\opt u(\cdot)$ is WOO iff for each admissible control $u(\cdot)$ holds (\ref{eq:limUwoo})
$$ 
\limsup_{T\rightarrow \infty}\int\limits_{0}^{T}\!\left(b + \sin(T-t)\right)\left(\opt u(t) - u(t)\right)\!\dd t \geq 0.
$$ 
For instance, control $\opt{u} \equiv 1$ is OO when $b \geq 1$.
When $b\in [0,1)$ control $\opt{u} \equiv 1$ is WOO.
\end{example}

The next example applies Corollary~\ref{c:limNS} to a linear problem, where  $\opt{J}_x(\tau,T)$ is unbounded as $T\to\infty$.

\begin{example}\label{ex:unbnd}
Let us maximize the following integral
$$\max_u \int_0^{\infty} x(t) \dd t, \quad \text{ s.t.:  } \dot x(t)  = u(t),\quad x(0) = 0,$$
where control $u(t) \leq 1$ for all $t\geq 0$.
 
We have the \emph{state-transition matrix} of the linearized system
$$Y^{\prime}(s)\, Y^{\prime}(t)^{-1} = 1$$
and vector-function $\opt{J}_x$, defined in (\ref{eq:Jx}) is unbounded in $T$:
\begin{align*}
\opt{J}_x(t,T) & =  \int_{t}^{T} \dd s = T-t.
\end{align*}
Hamilton-Pontryagin function takes the form
$ 
\Ham(x,u,t,\psi,\lambda) = \lambda \, x + \psi\, u
$ 
and its derivative w.r.t. $u$ at point $(x,u,t,\psi,\lambda) = (\opt x(t), \opt u(t),t,\opt{J}_x(t,T),1)$
$$
\frac{\partial \Ham}{\partial u}(\opt x(t), \opt u(t),t,\opt{J}_x(t,T),1) = T-t.
$$
Then control $\opt{u} \equiv 1$ is OO since for each admissible control $u(\cdot)$ holds (\ref{eq:limUoo})
$$ 
\liminf_{T\rightarrow \infty}\int\limits_{0}^{T}\!\left(T-t\right)\left(1 - u(t)\right)\!\dd t \geq 0.
$$ 
\end{example}

\begin{proposition}[Sufficient optimality conditions] 
\label{p:limS1}
Let $\Ham(x,u,t,\psi,1)$ be concave in $(x,u)$ for each $\psi$ and $t$, then admissible pair $(\hat{u}(\cdot),\hat{x}(\cdot))$.\\
1) is OO if for each admissible control $u(\cdot)$ holds
\begin{align}\label{eq:limUoo2}
&\liminf_{T\rightarrow \infty}\int\limits_{t_0}^{T}\!\left\langle\frac{\partial f}{\partial u}(\opt x(t), \opt u(t),t)\cdot \left(\opt{J}_x(t,T)-\psi(t)\right), \opt u(t) - u(t)\right\rangle\!\dd t \geq 0,
\end{align}
2) is WOO if for each admissible control $u(\cdot)$ holds
\begin{align}\label{eq:limUwoo2}
&\limsup_{T\rightarrow \infty}\int\limits_{t_0}^{T}\!\left\langle\frac{\partial f}{\partial u}(\opt x(t), \opt u(t),t)\cdot \left(\opt{J}_x(t,T)-\psi(t)\right), \opt u(t) - u(t)\right\rangle\!\dd t \geq 0,
\end{align}
where $\psi(t)$ is the adjoint variable for which maximum condition holds.
\end{proposition}

\begin{proof} Notice that $\Ham(x,u,t,\psi,\lambda)$ is linear w.r.t. $\psi$ as well as its derivatives. 
Hence 
\begin{align*}
&\frac{\partial \Ham}{\partial u}(\opt x(t), \opt u(t),t,\opt{J}_x(t,T),1) \nonumber\\
& = \frac{\partial \Ham}{\partial u}(\opt x(t), \opt u(t),t,\psi(t),1) + \frac{\partial f}{\partial u}(\opt x(t), \opt u(t),t)\cdot \left(\opt{J}_x(t,T)-\psi(t)\right).
\end{align*}
Since $\Ham(\opt x(t), \opt u(t),t,\psi(t),1)$ is the maximum in $u$, we have
$$\left\langle\frac{\partial \Ham}{\partial u}(\opt x(t), \opt u(t),t,\psi(t),1), \opt u(t) - u(t)\right\rangle \geq 0.$$
In the result of Lemma~\ref{p:limS} we can use the inequality
\begin{align*}&\left\langle\frac{\partial \Ham}{\partial u}(\opt x(t), \opt u(t),t,\opt{J}_x(t,T),1), \opt u(t) - u(t)\right\rangle\\ 
& \geq \left\langle\frac{\partial f}{\partial u}(\opt x(t), \opt u(t),t)\cdot \left(\opt{J}_x(t,T)-\psi(t)\right), \opt u(t) - u(t)\right\rangle.
\end{align*}%
 \end{proof}
 
 \begin{example}[The $q$-Theory of Investment, see, e.g. \cite{Acemoglu2009}, 7.8]\label{ex:TobinQ}
Let us maximize the following integral
$$\max_u \int\limits_0^{\infty}\! e^{- r \,t} \!\left(f(x(t)) - u(t) - \frac{c}{2}u^2(t)\right)\!\dd t, \,\, \text{ s.t.:  } \dot x(t)  = u(t) - \delta\, x(t),\,\, x(0) = x_0 > 0,$$
where $u(t)\in[u_1,u_2]$ is the investment intensity and $x(t)$ is the amount of capital, $f$ is a concave production function, interest rate $r > 0$, depreciation $\delta > 0$, coefficient $c > 0$, $\alpha\in(0,1)$.

We have the state-transition function of the linearized system
$$Y^{\prime}(s)\, Y^{\prime}(t)^{-1} = e^{-\delta\left(s-t\right)}$$
and vector-function $\opt{J}_x$, defined in (\ref{eq:Jx}) can be written via marginal Tobin\rq{}s $q$:
\begin{align*}
\opt{J}_x(t,T) & =   e^{- r \,t}\int\limits_{t}^{T} e^{-\left(\delta+r\right)\left(s-t\right) } \frac{\dd f(\opt x(s))}{\dd x} \dd s =  e^{- r \,t}  q(t) - e^{- \delta\left(T-t\right)} e^{- r \,T} q(T),
\end{align*}
where  $ q(t) = \int_{t}^{\infty} e^{-\left(\delta+r\right)\left(s-t\right)} \frac{\dd f(\opt x(s))}{\dd x} \dd s$.\\
Hamilton-Pontryagin function is concave w.r.t. $(x,u)$
$$ 
\Ham(x,u,t,\psi,\lambda) = \lambda \left(f(x) - u - \frac{c}{2}u^2\right) + \psi\left(u - \delta\, x\right)
$$ 
and derivative $\frac{\partial f}{\partial u}(\opt x(t), \opt u(t),t) =   1$.
Admissible control $\opt{u}(\cdot)$ is OO if for each admissible control $u(\cdot)$ holds (\ref{eq:limUoo2})
\begin{equation}\label{eq:limR}
\liminf_{T\rightarrow \infty} e^{- r \,T} q(T)\int\limits_{0}^{T} e^{-\delta\left(T-t\right)} \left(u(t) - \opt u(t)\right)\!\dd t \geq 0.
\end{equation}
We assume that control leads to stationary point $x(t)\to x^{*}$ and $q(t)\to q^{*}$ as $t\to\infty$, then limit in (\ref{eq:limR}) is zero, due to bounded control.
\end{example}

\begin{example}[\cite{Ramsey1928}]\label{ex:Ramsey}
Let us maximize the following integral
$$\max_u \int_0^{\infty} \log u(t) \dd t, \quad \text{ s.t.:  } \dot x(t)  = x^{\alpha}(t) - u(t) - \delta\, x(t),\,\, x(0) = x_0 > 0,$$
where control $u(t) > 0$ and state $x(t) > 0$ for all $t\geq 0$. $\delta > 0$, $\alpha\in(0,1)$.

Vector-function $\opt{J}_x \equiv 0$ due to its definition in (\ref{eq:Jx}).
Hamilton-Pontryagin function is concave w.r.t. $(x,u)$
$$ 
\Ham(x,u,t,\psi,\lambda) = \lambda \log u + \psi\left(x^{\alpha} - u - \delta\, x\right)
$$ 
and derivative $\frac{\partial f}{\partial u}(\opt x(t), \opt u(t),t) =   -1$.
Admissible control $\opt{u}(\cdot)$ is OO if for each admissible control $u(\cdot)$ holds (\ref{eq:limUoo2})
$$ 
\liminf_{T\rightarrow \infty}\int\limits_{0}^{T}\!\psi(t)\left(\opt u(t) - u(t)\right)\!\dd t \geq 0, 
$$ 
where $\psi(t) = 1/\opt u(t)$. This yields condition (\ref{eq:limUoo})
$$
\liminf_{T\rightarrow \infty}\int\limits_{0}^{T}\!\left(1 - \frac{u(t)}{\opt u(t)}\right)\!\dd t \geq 0.
$$
\end{example}
This sufficient condition is yet to be studied. Necessary conditions for this problem are considered in \cite{Belyakov2019}.

\section{Non-concavity and non-differentiability with respect to control}
We can relax conditions of concavity and differentiability of Hamilton-Pontryagin function with respect to control variable. 
The following proposition could be useful if state variable in Hamilton-Pontryagin function is additive-separable from control variable.
\begin{proposition}[Sufficient optimality conditions] 
\label{p:limS}
Let for almost all time instances $t \geq t_0$ there exists $T_1 = T_1(t) > t_0$ such that
$\Ham(x,u,t,\opt{J}_x(t,T),1)$ is concave in $x$ for each $T \geq T_1$,
where we take $\opt{J}_x(\tau,T)$ as defined (\ref{eq:Jx}). 
Then admissible pair $(\hat{u}(\cdot),\hat{x}(\cdot))$\\
1) is OO if for each admissible pair $(x(\cdot),u(\cdot))$ holds
\begin{align}\label{eq:limUoo3}
&\liminf_{T\rightarrow \infty}\!\int\limits_{t_0}^{T}\!\!\left(\Ham(x(t), \opt u(t),t,\opt{J}_x(t,T),1) - \Ham(x(t), u(t),t,\opt{J}_x(t,T),1)\right)\!\!\dd t \geq 0,
\end{align}
2) is WOO if for each admissible control $u(\cdot)$ holds
\begin{align}\label{eq:limUwoo3}
&\limsup_{T\rightarrow \infty}\!\int\limits_{t_0}^{T}\!\!\left(\Ham(x(t), \opt u(t),t,\opt{J}_x(t,T),1) - \Ham(x(t), u(t),t,\opt{J}_x(t,T),1)\right)\!\!\dd t \geq 0.
\end{align}
\end{proposition}
\begin{proof} Let us consider any admissible pair $(u(\cdot),x(\cdot))$, i.e. $ u(t)\in U$ and corresponding trajectory $x(t)\in X$ for all $t>t_0$. 
Due to concavity of $\Ham$ w.r.t. $x$  we have the following inequality
\begin{align}
& \Ham(x(t), u(t),t,\opt{J}_x(t,T),1) - \Ham(\opt x(t), \opt u(t),t,\opt{J}_x(t,T),1)  \nonumber\\
&= \Ham(x(t), u(t),t,\opt{J}_x(t,T),1) - \Ham( x(t), \opt u(t),t,\opt{J}_x(t,T),1)  \nonumber\\
&+ \Ham(x(t), \opt u(t),t,\opt{J}_x(t,T),1) - \Ham(\opt x(t), \opt u(t),t,\opt{J}_x(t,T),1)  \nonumber\\
& \leq \left\langle\frac{\partial \Ham}{\partial x}(\opt x(t), \opt u(t),t,\opt{J}_x(t,T),1), x(t) - \opt x(t)\right\rangle\nonumber\\
&  + \Ham(x(t), u(t),t,\opt{J}_x(t,T),1) - \Ham( x(t), \opt u(t),t,\opt{J}_x(t,T),1),\label{ineq:concx}
\end{align}
where $\frac{\partial \Ham}{\partial x}(\opt x(t), \opt u(t),t,\opt{J}_x(t,T),1) = -  \frac{\dd \opt{J}_x}{\dd t}(t,T)$ by definition of $\opt{J}_x(t,T)$ in (\ref{eq:Jx}).
Increment (\ref{eq:DJT}) in the value of the functional can be written as follows:
\begin{align*}
  \Delta J(T) & \geq \int\limits_{t_0}^{T}\!\!\left(\left\langle \frac{\dd \opt{J}_x}{\dd t}(t,T),  x(t) - \opt{x}(t) \right\rangle + \left\langle \opt{J}_x(t,T),  \frac{\dd x(t)}{\dd t} - \frac{\dd \opt{x}(t)}{\dd t} \right\rangle\right)\!\dd t  \nonumber\\
& + \int\limits_{t_0}^{T}\!\!\left(\Ham(x(t), \opt u(t),t,\opt{J}_x(t,T),1) - \Ham(x(t), u(t),t,\opt{J}_x(t,T),1)\right)\!\!\dd t,
\end{align*}
where the first integral is zero due to (\ref{eq:0}).
 \end{proof}
Similar to Corollary~\ref{c:limNS} if we require linearity w.r.t. $x$ instead of concavity, then inequality (\ref{ineq:concx}) becomes equality and sufficient conditions also become necessary.
\begin{corollary}[Nesessary and sufficient optimality conditions]\label{c:limNS3}
Let for almost all time instances $t \geq
t_0$ there exists $T_1 = T_1(t) > t_0$ such that
$\Ham(x,u,t,\opt{J}_x(t,T),1)$ is linear in $x$ for each $T \geq T_1$,
where we take $\opt{J}_x(\tau,T)$ as defined (\ref{eq:Jx}). 
Then admissible pair $(\hat{u}(\cdot),\hat{x}(\cdot))$\\
1) is OO iff for each admissible pair $(x(\cdot),u(\cdot))$ holds inequality (\ref{eq:limUoo3}),\\
2) is WOO iff for each admissible control $(x(\cdot),u(\cdot))$ holds  inequality (\ref{eq:limUwoo3}).
\end{corollary}
Consequently, conditions (\ref{eq:limUoo3})--(\ref{eq:limUwoo3}) lead to same results as (\ref{eq:limUoo})--(\ref{eq:limUwoo}) in linear Examples \ref{ex:oscil} and \ref{ex:unbnd}.

\section{Discussion}
We derive sufficient conditions (\ref{eq:limUoo})--(\ref{eq:limUwoo}) and (\ref{eq:limUoo2})--(\ref{eq:limUwoo2}) instead of (\ref{eq:SCoo})--(\ref{eq:SCwoo}) in the Mangasarian sufficiency theorem. New condition could be easier to check since we do not have to calculate all admissible trajectories of the state variable. We do not have also to find the \lq\lq{}true\rq\rq{} adjoint variable of the maximum principle. Instead we use particular adjoint  solution (\ref{eq:Jx}) that does not have to be bounded in $T$.  We extend conditions to the case where Hamiltonian is nether differentiable nor concave in control.

\section*{Acknowledgments}
This work was supported by 
Russian Science Foundation, grant 19-11-00223.

\section*{References}
\bibliography{biblio}

\begin{thebibliography}{10}
\expandafter\ifx\csname natexlab\endcsname\relax\def\natexlab#1{#1}\fi
\expandafter\ifx\csname url\endcsname\relax
  \def\url#1{\texttt{#1}}\fi
\expandafter\ifx\csname urlprefix\endcsname\relax\def\urlprefix{URL }\fi

\bibitem[{Acemoglu(2009)}]{Acemoglu2009}
Acemoglu, D., 2009. Introduction to Modern Economic Growth. Princeton
  University Press.

\bibitem[{Aseev(2017)}]{Aseev2017}
Aseev, S.~M., 2017. Existence of an optimal control in infinite-horizon
  problems with unbounded set of control constraints. Proceedings of the
  Steklov Institute of Mathematics (Suppl.) 297~(suppl. 1), 1--10.

\bibitem[{{Belyakov}(2015)}]{BelyakovCONDITION}
{Belyakov}, A.~O., Dec. 2015. {Necessary Conditions for Infinite Horizon
  Optimal Control Problems Revisited}. ArXiv e-prints.

\bibitem[{{Belyakov}(2019)}]{Belyakov2019}
{Belyakov}, A.~O., 2019. On necessary optimality conditions for ramsey-type
  problems. Ural Mathematical Journal 5~(1), 24--30.

\bibitem[{Carlson et~al.(1991)Carlson, Haurie, and
  Leizarowitz}]{Carlson1991infinite}
Carlson, D.~A., Haurie, A.~B., Leizarowitz, A., 1991. Infinite horizon optimal
  control. Springer-Verlag, Berlin, Heidelberg.

\bibitem[{Cartigny and Michel(2003)}]{Michel2003}
Cartigny, P., Michel, P., 2003. On a suffcient transversality condition for
  infinite horizon optimal control problems. Automatica 39, 1007--1010.

\bibitem[{Khlopin(2013)}]{Khlopin2013}
Khlopin, D., 2013. Necessity of vanishing shadow price in infinite horizon
  control problems. Journal of Dynamical and Control Systems 19~(4), 519--552.

\bibitem[{Pickenhain and Lykina(2006)}]{PickenhainLykina2006}
Pickenhain, S., Lykina, V., 2006. Sufficiency conditions for infinite horizon
  optimal control problems. In: Seeger, A. (Ed.), Recent Advances in
  Optimization. Springer Berlin Heidelberg, Berlin, Heidelberg, pp. 217--232.

\bibitem[{Ramsey(1928)}]{Ramsey1928}
Ramsey, F.~P., 1928. A mathematical theory of saving. The Economic Journal
  38~(152), 543--559.

\bibitem[{Seierstad and Syds{\ae}ter(1986)}]{Seierstad1986}
Seierstad, A., Syds{\ae}ter, K., 1986. Optimal control theory with economic
  applications. Elsevier North-Holland, Inc.

\end{thebibliography}
\end{document}